\documentclass[12pt]{article}
\usepackage{latexsym}
\usepackage{amsmath,amsfonts,amssymb,mathrsfs,graphicx}
\usepackage{amsthm}
\usepackage{color}
\usepackage{mathtools}
\usepackage{verbatim}
\title{Intermediate  dimension of images of sequences under fractional Brownian motion}
\author{Kenneth J. Falconer}
\date{}

\newcommand{\R}{\mathbb{R}}

\newcommand{\N}{\mathbb{N}}
\newcommand{\E}{\mathbb{E}}
\newcommand{\hd}{\dim_{\textup{H}}}

\newcommand{\ubd}{\overline{\dim}_{\textup{B}}}
\newcommand{\lbd}{\underline{\dim}_{\textup{B}}}
\newcommand{\uid}{\overline{\dim}_{\,\theta}}
\newcommand{\lid}{\underline{\dim}_{\,\theta}}

\newcommand{\be}{\begin{equation}}
\newcommand{\ee}{\end{equation}}

\newtheorem{theorem}{Theorem}[section]
\newtheorem{lemma}[theorem]{Lemma}

\theoremstyle{definition}

\theoremstyle{remark}

\numberwithin{equation}{section}

\begin{document}
\maketitle
\begin{center}
Mathematical Institute,	
University of St Andrews,\\
St Andrews,
Fife KY16 9SS,
UK
\medskip

E-mail: \texttt{kjf@st-andrews.ac.uk}
		 	
\end{center}

\begin{abstract}
\noindent We show that the almost sure $\theta$-intermediate dimension of the image of the set $F_p =\{0, 1,\frac{1}{2^p},\frac{1}{3^p},\ldots\}$ under index-$h$ fractional Brownian motion is $\frac{\theta}{ph+\theta}$, a value that is smaller than that given by directly applying the H\"{o}lder bound for fractional Brownian motion. In particular this establishes the box-counting dimension of these images.
\end{abstract}

\section{Introduction}\label{setting}
Intermediate dimensions were introduced in \cite{fafrke:2018} to interpolate between the Hausdorff dimension and box-counting dimensions of sets where these differ,  see the recent surveys \cite{FalSur,fra:survey:2019} for  surveys on intermediate dimensions and dimension interpolation. 
The lower and upper intermediate dimensions, $\lid E$ and  $\uid E$ of a set $E\subseteq\R^n$ depend on a parameter $\theta\in [0,1]$, with $\underline{\dim}_0 E = \overline{\dim}_0 E = \hd E$ and $\underline{\dim}_1 E = \lbd E$ and $\overline{\dim}_1 E = \ubd E$, where $\hd, \lbd$ and $\ubd$ denote Hausdorff, and lower and upper box-counting dimensions, respectively. Various properties of intermediate dimensions are established in \cite{Ban,fafrke:2018} with the intermediate dimensions reflecting the range of diameters of sets needed to get coverings that are efficient for estimating dimensions. In particular $\lid E$ and $\uid E$ are monotonically increasing in $\theta\in [0,1]$, are continuous except perhaps at $\theta = 0$, and are invariant under bi-Lipschitz mappings. Intermediate dimensions have been calculated for manysets which have differing Hausdorff and box-counting dimensions, including for sets of the form $F_p$ given in \eqref{fp} below in \cite{fafrke:2018}, as well as for  attractors of infinitely generated conformal iterated function systems \cite{BanFr}, spirals \cite{BFF2}, countable families of concentric spheres \cite{Tan}  and topologists' sine curves \cite{Tan}, with non-trivial bounds obtained for self-affine carpets \cite{fafrke:2018, Kol}. 

Specifically, for $E \subseteq \R^n$ and $0 \leq \theta \leq 1$, the  {\em lower intermediate dimension} of $E$ may be defined as
\begin{align}\label{lid}
\lid E =  \inf \big\{& s\geq 0  :  \mbox{ \rm for all $\epsilon >0$ and all $0< r_0 <1$,  there exists $0<r\leq r_0$} \\
 & \mbox{ \rm and a cover $ \{U_i\} $ of $E$ such that  $r^{1/\theta} \leq  |U_i| \leq r $ and 
 $\sum |U_i|^s \leq \epsilon$}  \big\}\nonumber
\end{align}
and the corresponding {\em upper intermediate dimension} by
\begin{align}\label{uid}
\uid E =  \inf \big\{& s\geq 0  :  \mbox{ \rm for all $\epsilon >0$ there exists $0< r_0 <1$ such that for all $0<r\leq r_0$,} \\
 & \mbox{ \rm there is a cover $ \{U_i\} $ of $E$ such that  $r^{1/\theta} \leq  |U_i| \leq r$ and 
 $\sum |U_i|^s \leq \epsilon$}  \big\},\nonumber
\end{align}
where $|U|$ denotes the diameter of a set $U \subseteq \R^n$. When $\theta = 0$ \eqref{lid} and \eqref{uid}
reduce to Hausdorff dimension, since there are  no lower bounds  on the diameters of covering sets. When $\theta=1$ all covering sets are forced to have the same diameter and we recover the lower and upper box-counting dimensions.

It is convenient to work with equivalent definitions of these intermediate dimensions in terms of the exponential behaviour of sums over covers.
For $E \subseteq \R^n$ bounded and non-empty,  $\theta  \in (0, 1]$, $r>0 $ and $s\in [0,n]$, define
\be\label{sums}
S_{r, \theta}^s(E) := \inf \Big\{\sum_i |U_i|^s : \{U_i\}_i\textnormal{ is a cover of }E \textnormal{ such that }r \leq |U_i| \leq r^\theta\,\,\textnormal{ for all } i\Big\}.
\ee
It is immediate that
\be\label{altlid}
\lid E =  \textnormal{ the unique } s\in [0,n] \textnormal{ such that  } \liminf\limits_{r \rightarrow 0} \frac{\log S_{r, \theta}^s(E)}{-\log r} =0
\ee
and 
$$
\uid E =  \textnormal{ the unique } s\in [0,n] \textnormal{ such that  } \limsup\limits_{r \rightarrow 0} \frac{\log S_{r, \theta}^s(E)}{-\log r} =0.
$$

For $p>0$ let 
\be\label{fp}
F_p \ = \ \bigg\{0,  \frac{1}{1^p}, \frac{1}{2^p},\dots, \frac{1}{3^p}\ldots \bigg\}.
\ee
It is shown in \cite{fafrke:2018}  that, for all $\theta \in [0,1]$,
$$
\dim_{\,\theta} F_p\ = \ \frac{\theta}{p+\theta}.
$$

Index-$h$ fractional Brownian motion $(0<h<1)$ is the stochastic process $B_h: \R^{\geq0} \to\R$ such that, almost surely, $B_h$ is continuous with $B_h(0) =0$, and the increments $B_h(x) - B_h(y)$  are stationary and Gaussian with mean 0 and variance $|x-y|^{2h}$, see, for example, \cite{EM,Falconer,Kah,MV,She}. There is a considerable literature on the dimensions of images of sets under stochastic processes, see \cite{Kah} for Hausdorff dimensions, and \cite{Xi} where the packing dimensions of images are expressed  in terms of dimension profiles.

Here we investigate the almost sure intermediate dimensions of $B_h(F_p)$, the image of $F_p$ under index-$h$ fractional Brownian motion. By a simple estimate on the intermediate dimensions of H\"{o}lder images of sets, or see \cite[Section 4]{Ban}, since  $B_h$ has almost sure H\"{o}lder exponent $h-\epsilon$ for all $\epsilon>0$, 
$$\dim_\theta B_h(F_p) \ \leq\  \frac{1}{h} \dim_\theta F_p \  =\ \frac{\theta}{h(p+\theta)};$$
however the actual value is smaller than this.

\begin{theorem}\label{main}
Let $B_h: \R\to \R  $ be index-$h$ fractional Brownian motion. Then almost surely, for all $\theta \in [0,1]$,
\be\label{ans}
\dim_\theta B_h(F_p)\ =\ \frac{\theta}{ph+\theta},
\ee
and in particular
$$
\dim_B B_h(F_p)\ =\ \frac{1}{ph+1}.
$$
\end{theorem}

We obtain the upper bound for $\dim_\theta B_h(F_p)$ by, for each $r$, covering the part of $B_h(F_p)$ near 0 by abutting intervals of lengths $r^\theta$ and the remaining points individually by intervals of length $r$. The lower bound uses a potential theoretic method, estimating an energy of the image under $B_h$ of the measure given by equal point masses on the points of $F_p$ between $1/(2M-1)^p$ and $1/M^p$.

`Intermediate dimension profiles'
 were introduced in \cite{BFF} to develop a general theory of intermediate dimensions including their behaviour under projections and this was developed for random images in \cite{Bur}. By \cite[Theorem 3.4]{Bur} 
$$\dim_\theta B_h(E)\ =\ \frac{1}{h}\dim_\theta^h E$$
where $\dim_\theta^s E$ is the $\theta$-dimension profile of a general compact $E\subseteq\R$, but this does not give an explicit value of $\dim_\theta B_h(F_p)$. Whilst this might be found by an awkward calculation of dimension profiles, our proof here of \eqref{ans} is self-contained.

\section{Proofs}\label{proofs}
We recall from \cite{BFF}  the energy kernels $\widetilde{\phi}_{r, \theta}^s$ on $\R^m$ defined  for $0<r<1, \theta\in [0,1]$ and $0<s\leq m$ by 
\be\label{kerhat}
\widetilde{\phi}_{r, \theta}^s(x)=
\begin{cases}
1 & |x| < r\\ 
\displaystyle{\bigg(\frac{r}{|x|}\bigg)^s} & r \leq |x| < r^\theta\\
0 & r^\theta \leq |x|
\end{cases} 
\ee
(here we only need the case of $m=1$).

The proof of the lower bound for \eqref{ans} uses the following three lemmas. The first is a slight variant of \cite[Lemma 4.3]{BFF} that relates the covering sums to energies with respect to the kernel 
$\widetilde{\phi}_{r, \theta}^s$. We write $\mathcal{M}(F)$ for the set of Borel probability measures supported by$F$.

\begin{lemma}\label{capacitylb}
Let $F\subseteq \R^m$ be compact, $\theta \in (0, 1]$, $0<r <1 $ and $0 \leq s \leq m$ and let $\mu \in \mathcal{M}(F)$. Then
\be\label{sbound}
S_{r, \theta}^s(F)\ \geq\ r^s\bigg[\int\!\!\int \widetilde{\phi}_{r, \theta}^s\big(x -  y\big)d\mu(x) d\mu(y)\bigg]^{-1}
\ee

\begin{proof}
Since $\widetilde{\phi}_{r, \theta}^s$ is lower semicontinuous, by standard potential theory there is an equilibrium measure $\mu_0$ for which  $\int\!\!\int \widetilde{\phi}_{r, \theta}^s\big(x -  x\big)d\mu(x) d\mu(y)$ attains its minimum, say,   
$$\int\!\!\int \phi_{r, \theta}^{s, n}(x - y) d\mu_0(x)d\mu_0(y) = \gamma.$$
Moreover, 
$$\int \widetilde{\phi}_{r, \theta}^s\big(x -  y\big)d\mu_0(y)\ \geq \ \gamma$$
for all $x\in F$, with equality if $x\in F_0$ for a set $F_0\subseteq F$ with  $\mu_0(F_0)=1$.

 If $r \leq \delta < r^\theta$ and $x \in F_0$ then using \eqref {kerhat},
\begin{equation}\label{zz}
\gamma = \int \widetilde{\phi}_{r, \theta}^s(x - y) d\mu(y) 
\geq \int \left(\frac{r}{\delta}\right)^s1_{B(0, \delta)}(x - y) d\mu(y) 
\geq \left(\frac{r}{\delta}\right)^s \mu(B(x, \delta)).
\end{equation}
Let $\{U_i\}_{i }$ be a finite cover of $F$ by sets of diameters $r \leq |U_i|< r^\theta$ and write $\mathcal{I} = \{i : U_i \cap F_0 \neq \emptyset \}$, so for each $i \in \mathcal{I}$ we may choose $x_i \in U_i \cap F_0$ such that $U_i \subseteq B(x_i, |U_i|)$. Then
\begin{equation*}
1 = \mu(F_0) \leq \sum\limits_{i \in \mathcal{I}} \mu(U_i) 
\leq  \sum\limits_{i \in \mathcal{I}} \mu(B(x_i, |U_i|))\leq r^{-s}\gamma \sum\limits_{i \in \mathcal{I}}|U_i|^s
\end{equation*}
by (\ref{zz}),  so
\begin{equation*}
\sum\limits_{i } |U_i|^s \geq r^s\gamma^{-1};
\end{equation*}
 taking the infimum over all such covers gives \eqref{sbound}. (Note that considering covers with $r \leq |U_i|< r^\theta$ makes no difference in the definition \eqref{sums}.)
\end{proof}
\end{lemma}

We next bound the expectation of  $\widetilde{\phi}_{r, \theta}^s$ evaluated on increments of fractional Brownian motion in terms of another kernel: 
\be\label{bounds}
{\psi}_{r, \theta}^{s}(x)
\  = \ \min\bigg\{1,\frac{r^{\theta(1-s) + s}}{|x|^{h}}\bigg\}.
\ee

\begin{lemma}\label{expect}
Let $B_h: \R\to \R  $ be index-$h$ fractional Brownian motion. Then
there is a constant $c$ depending only on $s$ such that for $0<r<1$, $0<\theta \leq 1$ and $0<s<1$,
$$
\E\big( \widetilde{\phi}_{r, \theta}^s(B_h(x) -  B_h(y))\big) 
\ \leq\ c \ {\psi}_{r, \theta}^{s}(x-y). 
$$

\end{lemma}
\begin{proof}
Since $B_h(x) -  B_h(y)$ has Gaussian density with mean $0$ and variance $|x-y|^{2h}$, 
$$\E\big( \widetilde{\phi}_{r, \theta}^s(B_h(x) -  B_h(y))\big) 
=
\frac{1}{\sqrt{2\pi}}\frac{1}{|x-y|^{h}}
\int_{-\infty}^\infty \widetilde{\phi}_{r, \theta}^s(t)  \exp\Big(\frac{-t^2}{2|x-y|^{2h}}\Big) dt.
$$
This is bounded above by 1 since $\widetilde{\phi}_{r, \theta}^s(t)\leq 1$ and the Gaussian has integral 1.
With $c_1=2/\sqrt{2\pi}$,
\begin{eqnarray*}
\E\big( \widetilde{\phi}_{r, \theta}^s(B_h(x) -  B_h(y))\big) 
&=& \frac{c_1}{|x-y|^{h}}
\int_{0}^\infty \widetilde{\phi}_{r, \theta}^s(t)  \exp\Big(\frac{-t^2}{2|x-y|^{2h}}\Big) dt\\
&=&\frac{c_1}{|x-y|^{h}}
\bigg[\int_{0}^r  \exp\Big(\frac{-t^2}{2|x-y|^{2h}}\Big) dt
+
\int_{r}^{r^\theta} \frac{r^s}{t^s}  \exp\Big(\frac{-t^2}{2|x-y|^{2h}}\Big) dt\bigg]\\
&\leq&\frac{c_1}{|x-y|^{h}}
\bigg[\int_{0}^r  dt+
\int_{r}^{r^\theta} \frac{r^s}{t^s} dt\bigg]\\
&\leq&\frac{c_1}{|x-y|^{h}}
\bigg[r+
\frac{1}{1-s}r^{s +\theta(1-s)} dt\bigg]\\
&\leq&c_2\frac{r^{s +\theta(1-s)}}{|x-y|^{h}}
\end{eqnarray*}
since $r<1$ and $s+\theta(1-s) \leq1$.
\end{proof} 

The next lemma estimates the energy of a measure on $F_p$ under the kernel ${\psi}_{r, \theta}^{s}$.

\begin{lemma}\label{energy}
Given $\theta, s \in (0,1]$ and $h\in (0,1)$ there is a number  $c>0$ such that for all $0<r<1$ there is a measure $\mu_r\in\mathcal{M}(F_p)$ such that 
$$
E(\mu_r) := \ \int\!\!\int {\psi}_{r, \theta}^{s}(x-y)d\mu_r(x) d\mu_r(y) \ \leq\ c\, r^{(\theta(1-s)+s)/(1+ph)}.
$$
\end{lemma}

\begin{proof}
Given $r$ we will find $M\equiv M(r)\in\N$ such that  the measure $\mu_r$ formed by placing a point mass of $1/M$ on each of the $M$ points of
$$F_p^M \ = \ \bigg\{ \frac{1}{(2M-1)^p}, \frac{1}{(2M-2)^p},\dots,  \frac{1}{M^p}\bigg\}\ \subseteq\ F_p$$
satisfies the conclusion. Note that if $M\leq k \leq 2M-1$, then 
$$\frac{p}{(2M)^{p+1}}\   \leq\  \frac{p}{(k+1)^{p+1}}\ <\  \Big|\frac{1}{k^p} -\frac{1}{(k+1)^p}\Big|\  < \   \frac{p}{k^{p+1}}\   \leq\  \frac{p}{M^{p+1}},$$
by a mean value theorem estimate. In particular, if $g$ is the minimum gap length between any pair of points of $F^M_p$, then $g\geq p/(2M)^{p+1}$ so any interval of length $R\geq g$ intersects at most $aM^{p+1}R$ points of $F^M_p$, where $a=2^{p+2}/p$. We estimate the energy 
$$E(\mu_r)\ = \ \int\!\! \int {\psi}_{r, \theta}^{s}(x-y)d\mu_r(x) d\mu_r(y) \ 
= \ \frac{1}{M^2}\sum_{M\leq i,j \leq 2M-1}{\psi}_{r, \theta}^{s}(x_i  - x_j)$$
where for convenience we write $x_i =1/i^p$.  Let $m$ be the greatest integer such that $2^m g  \leq M^{-p}$. Using  \eqref{bounds}, 
\begin{eqnarray*}
M^2 E(\mu_r) 
&=& \sum_{x_i  = x_j}\! \!{\psi}_{r, \theta}^{s}(x_i  - x_j)
\ +  \sum_{g \leq |x_i  - x_j|\leq 1/M^p}\! \!{\psi}_{r, \theta}^{s}(x_i  - x_j)\\
&=& 
M\ +\  \sum_{g\leq |x_i  - x_j|\leq 1/M^p}\frac{r^{\theta(1-s) + s}}{|x_i-x_j|^{h}}\\
&\leq&M\ +\ \sum_{k=0}^{m}   \sum_{2^k g  \leq |x_i  - x_j|\leq 2^{k+1} g}\frac{r^{\theta(1-s) + s}}{|x_i-x_j|^{h}}\\
&\leq&M\ +\ r^{\theta(1-s) + s}\sum_{k=0}^{m}  (2^k g)^{-h} MaM^{p+1}2^{k+1}g \\
&\leq&M\ +\ c_1 M^{p+2} r^{\theta(1-s) + s} g^{1-h}\, 2^{m(1-h)}\\
&\leq&M\ +\ c_1 M^{p+2}r^{\theta(1-s) + s} g^{1-h} (M^{-p}g^{-1})^{(1-h)}\\
&=&M\ +c_1 M^{2+ph}r^{\theta(1-s)+s} 
\end{eqnarray*}
where $c_1 $ is a constant depending only on $h$ and we have taken the dominant term of the geometric sum.
Thus 
\begin{eqnarray*}
E(\mu_r) &\leq& M^{-1}\ +\ c_1M^{ph}r^{\theta(1-s)+s}\\
&=& (2+c_12^{ph}) r^{(\theta(1-s)+s)/(1+ph)}
\end{eqnarray*}
on setting $M =\lceil r^{-(\theta(1-s)+s)/(1+ph)}\rceil \leq 2r^{-(\theta(1-s)+s)/(1+ph)}$ as $r<1$.
\end{proof} 

\noindent{\bf Proof of Theorem \ref{main}.}
\noindent The conclusion  is clear when $\theta = 0$, as $B_h(F_p)$ is countable so has Hausdorff dimension $0$, so assume  $\theta \in (0,1]$.

\noindent {\it Upper bound:} 
Let $0<\epsilon<h$. Index-$h$ fractional Brownian motion  satisfies 
$$|B_h(x)|\ \leq\ K t^{h-\epsilon} \qquad (0\leq x\leq 1)$$
almost surely for some $K<\infty$ see, for example, \cite{BarMou}.

Let $0<r<1$ and  $M$ be an integer and take a cover of $B_h(F_p)$ by intervals $\{U_i\}$ with
$r\leq |U_i|\leq r^\theta$ by covering each point  $B_h(1/k^p)\ (1\leq k \leq M)$ by an interval of length $r$ and covering 
$B_h\big([0,1/M^p]\big) \subseteq [-K M^{-p(h-\epsilon)},  K M^{-p(h-\epsilon)}]$ by abutting intervals of length $r^\theta$. Then
$$\sum_i|U_i|^s\ \leq\ M r^s +\bigg[\frac{2KM^{-p(h-\epsilon)}}{r^\theta}+1\bigg] (r^\theta)^s\ 
=\  Mr^s  + 2K r^{\theta(s-1)}M^{-p(h-\epsilon)}+ r^{s\theta}.$$
Setting $M = \lceil r^{(\theta (s-1)-s)/(1+p(h-\epsilon))}\rceil \geq 2$ gives
$$
\sum_i|U_i|^s\ \leq\ 2(1+K)r^{(s(p(h-\epsilon) + \theta) -\theta)/(1+p(h-\epsilon))}+ r^{s\theta} \to 0
$$
as $r\to 0$ provided that $s> \theta/ (p(h-\epsilon) + \theta)$. Taking $\epsilon$ arbitrarily small we conclude that 
\be\label{upbnd}
\uid B_h(F_p)\leq\displaystyle{\frac{\theta}{ph+\theta}}
\ee 
almost surely, by the definition \eqref{uid} of $\uid$. 

\noindent {\it Lower bound:} 
From Lemmas \ref{expect} and \ref{energy} and using Fubini's theorem, there is a number $c$ independent of $r$ such that for all $0<r<1$ there is a measure $\mu_r$ on $F_p$ such that 
$$\E\bigg(\int\!\!\int \widetilde{\phi}_{r, \theta}^s\big(B_h(x) -  B_h(y)\big)d\mu_r(x)d\mu_r(y)\bigg) \ 
\leq c\ r^{(\theta(1-s)+s)/(1+ph)}.$$
For $\epsilon >0$, setting $r=2^{-k}, k\in \N$, and summing,
$$\E\bigg(\sum_{k=1}^\infty 2^{k[(\theta(1-s)+s)/(1+ph)-\epsilon]}\int\!\!\int \widetilde{\phi}_{2^{-k}, \theta}^s\big(B_h(x) -  B_h(y)\big)d\mu_{2^{-k}}(x)d\mu_{2^{-k}}(y)\bigg) \ 
\leq c\ \sum_{k=1}^\infty2^{-k\epsilon}\ <\infty.$$
Hence, almost surely there exists a random $K<\infty$ such that 
$$ \int\!\!\int \widetilde{\phi}_{r, \theta}^s\big(B_h(x) -  B_h(y)\big)d\mu_{r}(x)d\mu_{r}(y)
\ \leq\ K r^{(\theta(1-s)+s)/(1+ph)-\epsilon}$$
for $r=2^{-k}$ for all $k\in\N$ and thus for all $0<r<1$ with a modified $K$, noting that the two sides change only by a bounded ratio on replacing $r$ by $2^{-k}$ for the least $k$ such that $2^{-k} \leq r$. Writing $\widetilde{\mu}_r$ for the image measure of $\mu_r$ under $B_h$, so $\widetilde{\mu}_r$ is supported by $B_h(F_p)$ (in the notation of Lemma \ref{energy} $\widetilde{\mu}_r$ consists of a mass of $1/M$ on each point $B_r(x_i)$), this becomes 
$$ \int\!\!\int \widetilde{\phi}_{r, \theta}^s\big(u -  v\big)d\widetilde{\mu}_{r}(u)d\widetilde{\mu}_{r}(v)
\ \leq\ K r^{(\theta(1-s)+s)/(1+ph)-\epsilon}.$$
Thus, by Lemma \ref{capacitylb}, almost surely there is a $K<\infty$ such that for all $0<r<1$, 
$$S_{r, \theta}^s(B_h(F_p))\ \geq\ K^{-1} r^{ s-(\theta(1-s)+s)/(1+ph)+\epsilon}$$
so
$$\liminf\limits_{r \rightarrow 0} \frac{\log S_{r, \theta}^s(B_h(F_p))}{-\log r}\ \geq\  
-s+\frac{\theta(1-s)+s}{(1+ph)}-\epsilon \ =\ \frac{\theta(1-s)-sph}{(1+ph)}-\epsilon
$$
and this remains true almost surely on setting $\epsilon =0$. Hence 
$\liminf\limits_{r \rightarrow 0} \log S_{r, \theta}^s(B_h(F_p))/-\log r \geq 0$ if $\theta(1-s)+sph =0$, that is if
$s= \theta/(ph+\theta)$, so by \eqref{altlid}
\be\label{lobnd}
\lid B_h(F_p)\ \geq\ \displaystyle{\frac{\theta}{ph+\theta}}.
\ee
Combined with \eqref{upbnd} this  gives \eqref{ans} almost surely for each $\theta \in [0,1]$. Thus, almost surely, \eqref{ans} holds for all rational $\theta\in [0,1]$ simultaneously, and so, since box dimensions are continuous for  $\theta\in (0,1]$, for all  $\theta\in [0,1]$ simultaneously.
\hfill$\Box$ 

Finally we remark that a similar approach can be used to find or estimate the intermediate and box-counting dimensions of fractional Brownian images of sets defined by other sequences tending to 0. For example, let $f: [1,\infty)\to \R^+$ be a decreasing function and let $F= \{0,f(1), f(2), \ldots\}$. Then if $f(x) = O(x^{-p})$ the upper bound argument gives
\be
\uid B_h(F)\leq\displaystyle{\frac{\theta}{ph+\theta}}
\ee 
almost surely. If we also assume that $f$ is differentiable with non-increasing absolute derivative $|f'|$ then $F$ has `decreasing gaps', that is $f(k)-f(k+1)$ is non-increasing, and if
$$\frac{f(x)^{1-h}}{|f'(2x)|} = O(x^{1+ph})$$
a similar energy argument gives that almost surely, for all $\theta\in [0,1]$,
\be
\lid B_h(F)\ \geq\ \displaystyle{\frac{\theta}{ph+\theta}}.
\ee

\section*{Acknowledgements}
The author thanks Amlan Banaji, Stuart Burrell, Jonathan Fraser and Istvan Kolossv\'{a}ry for many discussions related to intermediate dimensions.


\begin{thebibliography}{99}

\bibitem{Ban}
A. Banaji, Generalised intermediate dimensions,  arxiv: 2011.08613
 
\bibitem{BanFr}
A. Banaji and J.M. Fraser, Intermediate dimensions of infinitely generated attractors,  arxiv: 2104.15133
 
\bibitem{BarMou}
{D. Baraka and T. Mountford}, A law of the iterated logarithm for fractional Brownian motions, \emph{S\'{e}minaire de Probabilit\'{e}s  XLI}, (Eds C. Donati-Martin, M. \'{E}mery, A. Rouault, C. Stricker),  {\em Springer Lecture Notes in Mathematics} {\bf 1934}(2015), 161-179.

\bibitem{Bur}
S. Burrell, Dimensions of fractional Brownian images,  arxiv: 2002.03659

\bibitem{BFF}
S. Burrell, K.J. Falconer and J.M. Fraser, Projection theorems for intermediate dimensions, \emph{J. Fractal Geom.} {\bf 8}(2021), 95--116. 

\bibitem{BFF2}
S. Burrell, K.J. Falconer and J.M. Fraser, The fractal structure of elliptical polynomial spirals, arxiv: 2008.08539

\bibitem{EM}
{P. Embrechts and M. Maejima}, {\em Selfsimilar Processes}, Princeton University Press, 2002.

\bibitem{Falconer}
{K.J. Falconer}, {\em Fractal Geometry: Mathematical Foundations and Applications}, John Wiley, 3rd. ed., 2014.

\bibitem{FalSur}
{K.J. Falconer}, Intermediate dimensions -- a survey, {\em Thermodynamic Formalism} (eds. Mark Pollicott and Sandro Vaienti), {\em Springer Lecture Notes in Mathematics} {\bf 2290}(2021), arxiv: 2011.04363



\bibitem{fafrke:2018}
{K.J. Falconer, J.M Fraser and T. Kempton}, Intermediate dimensions,  \emph{Math. Zeit.} {\bf 296}(2020), 813--830.





\bibitem{fra:survey:2019}
{J.M. Fraser}, Interpolating between dimensions, in \emph{Fractal Geometry and Stochastics VI}, (Eds: U. Freiberg, B. Hambly, M. Hinz, S. Winter), Birkh\"auser, {\em Progress in Probability}, {\bf 76}(2021) 3--24.




\bibitem{Kah}
{J.P. Kahane}, {\em Some Random Series of Functions}, Cambridge University Press, 1985.

\bibitem{Kol} 
{I. Kolossv\'{a}ry}, On the intermediate dimensions of Bedford-McMullen
carpets, arxiv: 2006.14366

\bibitem{MV}
{B.B. Mandelbrot and J.W. Van Ness}, Fractional Brownian motion, fractional noises and applications, {\em SIAM Review} {\bf 10}(1968), 422-437.





\bibitem{She} 
{G. Shevchenko}, Fractional Brownian motion in a nutshell, in {\em Analysis of Fractional Stochastic Processes}, 1560002,  {\em International Journal of Modern Physics Conference Series} {\bf 36}, World Scientific, 2015.

\bibitem{Tan}
J.T. Tan, On the intermediate dimensions of concentric spheres and related sets, arxiv: 2008.10564

\bibitem{Xi}
Y. Xiao, Packing dimension of the image of fractional Brownian motion {\em Statist. Probab. Lett.} {\bf 33}(1997), 379--387.

\end{thebibliography}
\end{document}